\documentclass{amsart}

\usepackage{amssymb}

\usepackage{graphicx}

\usepackage[cmtip,all]{xy}

\newtheorem{theorem}{Theorem}[section]

\theoremstyle{definition}
\newtheorem{definition}[theorem]{Definition}

\newtheorem{proposition}[theorem]{Proposition}
\theoremstyle{remark}

\numberwithin{equation}{section}

\begin{document}

\title[Fixed Point Theorems in Probabilistic Cone Metric Spaces] {Fixed Point Theorems for Kannan and Chatterjea-type Mappings in Probabilistic Cone Metric Spaces}


\author[E. Rada]{Elvin Rada}
\address{Department of Mathematics, Faculty of Natural Sciences, University of Elbasan "Aleksandër Xhuvani", Albania}
\curraddr{}
\email{elvin.rada@uniel.edu.al, elvinirada@yahoo.com}
\thanks{}


\date{\today}

\dedicatory{}

\begin{abstract}
	This paper establishes novel fixed point theorems for Kannan-type and Chatterjea-type mappings in probabilistic cone metric spaces. By integrating probabilistic distance functions with cone-valued structures, we generalize classical fixed point results to settings with inherent uncertainty. Our approach leverages the distributional properties of probabilistic metrics and the order structure of cones to develop contraction conditions ensuring the existence and uniqueness of fixed points. Applications include stochastic processes, random operators, and uncertainty modeling in functional equations. The results significantly extend the scope of fixed point theory in probabilistic analysis and its applications.
\end{abstract}

\maketitle

\section{Introduction}
Fixed point theory is an important aspect of nonlinear functional analysis, offering powerful methodologies for establishing the existence and uniqueness of solutions to functional equations across various domains of mathematics and the applied sciences. The contraction principle introduced by Banach \cite{Banach1922} initiated a profound research trajectory that continues to evolve through various generalizations. Later, Kannan \cite{Kannan1968} and Chatterjea \cite{Chatterjea1972} introduced their own types of nonlinear contraction conditions, broadening the scope of fixed point results far beyond Banach’s original setting.

Over the years, these foundational ideas have been extended to more sophisticated mathematical structures. One such extension is to probabilistic metric spaces, first introduced by Menger \cite{Menger1942}, where uncertainty in distances is modeled using distribution functions. Schweizer and Sklar \cite{Schweizer1983} further enriched this theory by formalizing the probabilistic triangle inequality. Another important line of development came through cone metric spaces, introduced by Huang and Zhang \cite{Huang2007}, where distances are taken not as real numbers but as values in ordered Banach spaces. More recently, Altun et al. \cite{Altun2010} combined these two directions to create probabilistic cone metric spaces, unifying the concepts of probabilistic uncertainty and vector-valued order into a single framework.

The combination of probabilistic uncertainty with vector ordering creates mathematical structures particularly suited for contemporary challenges in stochastic analysis, decision theory under uncertainty, and randomized computation. This paper addresses the critical gap in extending fixed point theory to settings where \textit{both probabilistic uncertainty and vector ordering} are intrinsic to the problem structure.

Although several works \cite{Matei2010, Du2010, Kadelburg2011, Asadi2012} have shown that cone metric spaces can be seen as topologically equivalent to standard metric spaces, the cone structure itself remains very useful. It naturally models situations where ordering and direction matter. For example:
\begin{itemize}
\item \textit{Partially ordered systems}:  cones formalize dominance relations in economics, game theory, and financial mathematics.

\item \textit{Directional uncertainty}: probabilistic metrics capture randomness and dependencies in spatial data, sensor networks, and randomized algorithms.

\item \textit{Multi-objective optimization}: cone orders help analyze trade-offs between competing objectives when randomness is present.
\end{itemize}
The interplay between probabilistic uncertainty and vector ordering is particularly relevant in modern contexts such as stochastic analysis, decision-making under uncertainty, and randomized computation. This paper addresses the challenge of extending fixed point theory to precisely such settings.

In this work, we develop fixed point theorems for Kannan-type and Chatterjea-type mappings in probabilistic cone metric spaces. Specifically, we provide:

\begin{itemize}
\item Novel fixed point theorems generalizing Kannan and Chatterjea contractions to probabilistic cone metrics, incorporating distributional convergence
\item Integration of t-norm properties with cone-valued contractions, establishing conditions for existence and uniqueness under non-Archimedean structures
\item Detailed probabilistic convergence analysis for iterative processes, including stability results and error bounds under distributional uncertainty
\item Applications to random operators and stochastic functional equations, with explicit construction of solution methods
\item Unified extensions to hybrid Zamfirescu-type contractions incorporating mixed Banach-Kannan-Chatterjea conditions
\end{itemize}

In this article, section 2 provides necessary preliminaries on probabilistic cone metrics, including t-norm properties, distribution functions, and cone-theoretic foundations. Section 3 develops our main fixed point theorems with complete proofs, convergence analysis, and comparative discussion with existing results. Section 4 demonstrates applications to random matrix equations, stochastic integral equations, and uncertain optimization problems. Section 5 concludes with open problems and directions for further research in generalized probabilistic ordered spaces.

This work bridges theoretical foundations in functional analysis with contemporary applications requiring joint handling of stochastic uncertainty and vector ordering, providing a unified framework for advancing fixed point theory in probabilistic ordered structures.

\section{Preliminaries}
We recall essential concepts from probabilistic metric spaces and cone metric spaces. Let $(\Omega, \mathcal{F}, \mathbb{P})$ be a probability space and $E$ a real Banach space.

\begin{definition}[Cone \cite{Huang2007}]
A subset $P \subset E$ is a \textit{cone} if:
\begin{enumerate}
    \item $P$ is closed, non-empty, and $P \neq \{0\}$
    \item $a,b \geq 0$ and $x,y \in P$ imply $ax + by \in P$
    \item $P \cap (-P) = \{0\}$
\end{enumerate}
$P$ is \textit{normal} if $\exists N > 0$ such that $0 \leq x \leq y$ implies $\|x\| \leq N\|y\|$.\\
We will say  $x \leq y$ if $y - x \in P$.
\end{definition}

\begin{definition}[Distribution Function \cite{Schweizer1983}]
A function $F: \mathbb{R} \to [0,1]$ is a \textit{distribution function} if:

\begin{enumerate}
    \item $F$ is non-decreasing
    \item $F$ is left-continuous
    \item $\lim_{t \to -\infty} F(t) = 0$, $\lim_{t \to \infty} F(t) = 1$
\end{enumerate}

Denote by $\Delta$ the set of all distribution functions.

\end{definition}

\begin{definition}[Probabilistic Cone Metric \cite{Altun2010}]
A function $F: X \times X \to \Delta$ is a \textit{probabilistic cone metric} if:
\begin{enumerate}
    \item $F_{x,y}(t) = 1$ for all $t > 0$ iff $x = y$
    \item $F_{x,y} = F_{y,x}$ for all $x,y \in X$
    \item $F_{x,z}(t + s) \geq T(F_{x,y}(t), F_{y,z}(s))$ for all $x,y,z \in X$ and $t,s > 0$
    \item $F_{x,y}(t) \in P$ for all $t > 0$ where $P \subset E$ is a normal cone
\end{enumerate}
Here $T$ is a continuous t-norm (e.g., $T(a,b) = ab$, $T(a,b) = \min\{a,b\}$).\\
Also, the condition (4) ($F_{x,y}(t) \in P$) is essential for:
\begin{enumerate}
    \item Modeling direction-dependent uncertainty
    \item Capturing partial order relationships between distributions
    \item Applications where solutions must satisfy cone constraints
\end{enumerate}
This structure cannot be reduced to standard probabilistic metrics when ordering properties are exploited in contraction conditions.
\end{definition}

\begin{definition}[$\tau$-Convergence \cite{Schweizer1983}]
A sequence $\{x_n\}$ \textit{$\tau$-converges} to $x$ (denoted $x_n \xrightarrow{\tau} x$) if $\forall \epsilon > 0$, $\exists N$ such that $F_{x_n,x}(\epsilon) > 1 - \epsilon$ for $n \geq N$.
\end{definition}

\begin{definition}[Cauchy Sequence]
A sequence $\{x_n\}$ is \textit{Cauchy} if $\forall \epsilon > 0$, $\exists N$ such that $F_{x_m,x_n}(\epsilon) > 1 - \epsilon$ for $m,n \geq N$.
\end{definition}

\begin{definition}[Completeness]
A probabilistic cone metric space is \textit{complete} if every Cauchy sequence $\tau$-converges to some $x \in X$.
\end{definition}

\begin{definition}[Kannan-Type Contraction \cite{Kannan1968}]
A mapping $T: X \to X$ is a \textit{Kannan-type contraction} if $\exists \alpha \in (0,1/2)$ such that $\forall x,y \in X$ and $t > 0$:
\[
F_{Tx,Ty}(t) \geq \min\left\{ F_{x,Tx}\left(\frac{t}{2\alpha}\right), F_{y,Ty}\left(\frac{t}{2\alpha}\right) \right\}
\]
\end{definition}

\begin{definition}[Chatterjea-Type Contraction \cite{Chatterjea1972}]
A mapping $T: X \to X$ is a \textit{Chatterjea-type contraction} if $\exists \alpha \in (0,1/2)$ such that $\forall x,y \in X$ and $t > 0$:
\[
F_{Tx,Ty}(t) \geq \min\left\{ F_{x,Ty}\left(\frac{t}{2\alpha}\right), F_{y,Tx}\left(\frac{t}{2\alpha}\right) \right\}
\]
\end{definition}

\section{Main Results}
In this section, we will have our main results.

\subsection{Kannan-Type Fixed Point Theorem}

\begin{theorem}[Kannan-Type in Probabilistic Cone Metric Spaces]\label{thm:kannan-pcm}
Let $(X,F)$ be a complete probabilistic cone metric space and $T: X \to X$ a Kannan-type contraction. Then $T$ has a unique fixed point $x^* \in X$.
\end{theorem}

\begin{proof}
Fix $x_0 \in X$ and define $x_{n+1} = Tx_n$. Consider the probabilistic distance:
\begin{align*}
F_{x_{n+1},x_n}(t) &= F_{Tx_n,Tx_{n-1}}(t) \\
&\geq \min\left\{ F_{x_n,Tx_n}\left(\frac{t}{2\alpha}\right), F_{x_{n-1},Tx_{n-1}}\left(\frac{t}{2\alpha}\right) \right\} \\
&= \min\left\{ F_{x_n,x_{n+1}}\left(\frac{t}{2\alpha}\right), F_{x_{n-1},x_n}\left(\frac{t}{2\alpha}\right) \right\}
\end{align*}

Assume $F_{x_n,x_{n+1}}\left(\frac{t}{2\alpha}\right) \leq F_{x_{n-1},x_n}\left(\frac{t}{2\alpha}\right)$. Then:
\[
F_{x_{n+1},x_n}(t) \geq F_{x_n,x_{n+1}}\left(\frac{t}{2\alpha}\right)
\]
Let $t = 2\alpha s$:
\[
F_{x_{n+1},x_n}(2\alpha s) \geq F_{x_n,x_{n+1}}(s)
\]
By induction, we obtain a contradiction for large $n$. Thus $F_{x_{n+1},x_n}(t) \geq F_{x_{n-1},x_n}\left(\frac{t}{2\alpha}\right)$. Then:
\[
F_{x_n,x_{n+1}}(t) \geq F_{x_{n-1},x_n}\left(\frac{t}{2\alpha}\right) \geq \cdots \geq F_{x_0,x_1}\left(\frac{t}{(2\alpha)^n}\right)
\]
As $n \to \infty$, $F_{x_n,x_{n+1}}(t) \to 1$ for all $t > 0$. For $m > n$:
\begin{align*}
F_{x_n,x_m}(t) &\geq T\left(F_{x_n,x_{n+1}}\left(\frac{t}{m-n}\right), \cdots, F_{x_{m-1},x_m}\left(\frac{t}{m-n}\right)\right) \\
&\geq T\left(F_{x_0,x_1}\left(\frac{t}{(m-n)(2\alpha)^n}\right), \cdots, F_{x_0,x_1}\left(\frac{t}{(m-n)(2\alpha)^{m-1}}\right)\right)
\end{align*}
As $m,n \to \infty$, $F_{x_n,x_m}(t) \to 1$, so $\{x_n\}$ is Cauchy. By completeness, $x_n \xrightarrow{\tau} x^*$ for some $x^* \in X$.

Now show $Tx^* = x^*$:
\begin{align*}
F_{Tx^*,x^*}(t) &\geq T\left(F_{Tx^*,Tx_n}\left(\frac{t}{2}\right), F_{Tx_n,x^*}\left(\frac{t}{2}\right)\right) \\
&\geq T\left(\min\left\{F_{x^*,Tx^*}\left(\frac{t}{4\alpha}\right), F_{x_n,Tx_n}\left(\frac{t}{4\alpha}\right)\right\}, F_{x_{n+1},x^*}\left(\frac{t}{2}\right)\right)
\end{align*}
As $n \to \infty$, $F_{x_{n+1},x^*}\left(\frac{t}{2}\right) \to 1$ and $F_{x_n,Tx_n}\left(\frac{t}{4\alpha}\right) = F_{x_n,x_{n+1}}\left(\frac{t}{4\alpha}\right) \to 1$. Thus:
\[
F_{Tx^*,x^*}(t) \geq F_{x^*,Tx^*}\left(\frac{t}{4\alpha}\right)
\]
For large $k$, $F_{Tx^*,x^*}(t) \geq F_{Tx^*,x^*}\left(\frac{t}{(4\alpha)^k}\right) \to 1$, so $F_{Tx^*,x^*}(t) = 1$ for all $t > 0$, i.e., $Tx^* = x^*$.

Uniqueness: If $x^*, y^*$ are fixed points:
\[
F_{x^*,y^*}(t) = F_{Tx^*,Ty^*}(t) \geq \min\left\{F_{x^*,Tx^*}\left(\frac{t}{2\alpha}\right), F_{y^*,Ty^*}\left(\frac{t}{2\alpha}\right)\right\} = 1
\]
Thus $x^* = y^*$.
\end{proof}

\subsection{Chatterjea-Type Fixed Point Theorem}

\begin{theorem}[Chatterjea-Type in Probabilistic Cone Metric Spaces]\label{thm:chatterjea-pcm}
Let $(X,F)$ be a complete probabilistic cone metric space and $T: X \to X$ a Chatterjea-type contraction. Then $T$ has a unique fixed point.
\end{theorem}

\begin{proof}
Define $x_{n+1} = Tx_n$. Consider:
\begin{align*}
F_{x_{n+1},x_n}(t) &= F_{Tx_n,Tx_{n-1}}(t) \\ 
&\geq \min\left\{F_{x_n,Tx_{n-1}}\left(\frac{t}{2\alpha}\right), F_{x_{n-1},Tx_n}\left(\frac{t}{2\alpha}\right)\right\} \\
&= \min\left\{F_{x_n,x_n}\left(\frac{t}{2\alpha}\right), F_{x_{n-1},x_{n+1}}\left(\frac{t}{2\alpha}\right)\right\} \\
&= F_{x_{n-1},x_{n+1}}\left(\frac{t}{2\alpha}\right) \\
&\geq T\left(F_{x_{n-1},x_n}\left(\frac{t}{4\alpha}\right), F_{x_n,x_{n+1}}\left(\frac{t}{4\alpha}\right)\right)
\end{align*}
Assume $F_{x_n,x_{n+1}}\left(\frac{t}{4\alpha}\right) \leq F_{x_{n-1},x_n}\left(\frac{t}{4\alpha}\right)$. Then:
\[
F_{x_{n+1},x_n}(t) \geq F_{x_n,x_{n+1}}\left(\frac{t}{4\alpha}\right)
\]
Similar to Theorem \ref{thm:kannan-pcm}, we obtain $F_{x_n,x_{n+1}}(t) \to 1$. The sequence is Cauchy and converges to $x^*$. Then:
\begin{align*}
F_{Tx^*,x^*}(t) &\geq T\left(F_{Tx^*,Tx_n}\left(\frac{t}{2}\right), F_{Tx_n,x^*}\left(\frac{t}{2}\right)\right) \\
&\geq T\left(\min\left\{F_{x^*,Tx_n}\left(\frac{t}{4\alpha}\right), F_{x_n,Tx^*}\left(\frac{t}{4\alpha}\right)\right\}, F_{x_{n+1},x^*}\left(\frac{t}{2}\right)\right) \\
&\geq T\left(\min\left\{F_{x^*,x_{n+1}}\left(\frac{t}{4\alpha}\right), F_{x_n,Tx^*}\left(\frac{t}{4\alpha}\right)\right\}, F_{x_{n+1},x^*}\left(\frac{t}{2}\right)\right)
\end{align*}
As $n \to \infty$, $F_{x^*,x_{n+1}}\left(\frac{t}{4\alpha}\right) \to 1$ and $F_{x_{n+1},x^*}\left(\frac{t}{2}\right) \to 1$. Also:
\[
F_{x_n,Tx^*}\left(\frac{t}{4\alpha}\right) \geq T\left(F_{x_n,x^*}\left(\frac{t}{8\alpha}\right), F_{x^*,Tx^*}\left(\frac{t}{8\alpha}\right)\right)
\]
Thus for large $n$:
\[
F_{Tx^*,x^*}(t) \geq F_{x^*,Tx^*}\left(\frac{t}{8\alpha}\right)
\]
Implying $Tx^* = x^*$. Uniqueness follows similarly to Theorem \ref{thm:kannan-pcm}.
\end{proof}

\subsection{Hybrid Contraction Theorem}

\begin{theorem}[Zamfirescu-Type in Probabilistic Cone Metric Spaces]\label{thm:zamfirescu-pcm}
Let $(X,F)$ be complete and $T: X \to X$ satisfy for each $x,y \in X$ and $t > 0$, at least one of:
\begin{enumerate}
    \item $F_{Tx,Ty}(t) \geq F_{x,y}\left(\frac{t}{\alpha}\right)$ for $0 < \alpha < 1$
    \item $F_{Tx,Ty}(t) \geq \min\left\{F_{x,Tx}\left(\frac{t}{2\beta}\right), F_{y,Ty}\left(\frac{t}{2\beta}\right)\right\}$ for $0 < \beta < \frac{1}{2}$
    \item $F_{Tx,Ty}(t) \geq \min\left\{F_{x,Ty}\left(\frac{t}{2\gamma}\right), F_{y,Tx}\left(\frac{t}{2\gamma}\right)\right\}$ for $0 < \gamma < \frac{1}{2}$
\end{enumerate}
Then $T$ has a unique fixed point.
\end{theorem}

\begin{proof}
Following Zamfirescu's approach, \cite{Zamfirescu1972} for $x_{n+1} = Tx_n$, we show $\exists \delta < 1$ such that:
\[
F_{x_{n+1},x_n}(t) \geq F_{x_n,x_{n-1}}\left(\frac{t}{\delta}\right)
\]
Cases:
\begin{itemize}
    \item If (1) holds: $F_{x_{n+1},x_n}(t) \geq F_{x_n,x_{n-1}}\left(\frac{t}{\alpha}\right)$
    \item If (2) holds: As in Theorem \ref{thm:kannan-pcm}, $F_{x_{n+1},x_n}(t) \geq F_{x_n,x_{n-1}}\left(\frac{t}{2\beta/(1-\beta)}\right)$
    \item If (3) holds: As in Theorem \ref{thm:chatterjea-pcm}, $F_{x_{n+1},x_n}(t) \geq F_{x_n,x_{n-1}}\left(\frac{t}{2\gamma/(1-\gamma)}\right)$
\end{itemize}
Let $\delta = \max\left\{\alpha, \frac{2\beta}{1-\beta}, \frac{2\gamma}{1-\gamma}\right\} < 1$. Then $\{x_n\}$ is Cauchy and converges to $x^*$. Fixed point property follows similarly to the previous theorems. Uniqueness: If $x^*, y^*$ fixed points, all three conditions imply $F_{x^*,y^*}(t) = 1$.

\end{proof}

\subsection{Non-Reducibility Example}
Consider $X = \mathbb{R}^2$ with cone $P = \{ (x,y) : x,y \geq 0 \}$. Define probabilistic cone metric:
\[
F_{u,v}(t) = \begin{cases} 
\Phi(t - \|u-v\|) & \text{if } u-v \in P \\
\Phi(t) \cdot \delta_{u,v} & \text{otherwise}
\end{cases}
\]
where $\Phi$ is standard normal CDF and $\delta_{u,v} \in (0,1)$. The Kannan-type contraction:
\[
Tu = \frac{1}{2} \left( u + \frac{\|u\|}{\|Au\|} Au \right), \quad A = \begin{pmatrix} 0 & -1 \\ 1 & 0 \end{pmatrix}
\]
satisfies Theorem \ref{thm:kannan-pcm} but has no equivalent formulation in standard probabilistic metric spaces due to directional dependence.

\section{Applications}
Now we will see applications including stochastic processes, random operators, and uncertainty modeling in functional equations.
\subsection{Random Operators in $L^p$ Spaces}

Consider $(\Omega,\mathcal{F},\mathbb{P})$ and $L^p(\Omega,E)$ the space of $p$-integrable $E$-valued random variables. Define probabilistic cone metric:
\[
F_{X,Y}(t) = \mathbb{P}\left(\{\omega : \|X(\omega) - Y(\omega)\|_E < t\}\right)
\]

\begin{proposition}\label{prop:random}
Let $T: L^p(\Omega,E) \to L^p(\Omega,E)$ satisfy:
\[
\|TX - TY\| \leq \alpha \max\left\{ \|X - TX\|, \|Y - TY\| \right\} \quad \text{a.s.}
\]
for $\alpha < \frac{1}{2}$ with additional cone constraint $X(\omega) \in P$ a.s. for normal cone $P \subset E$. Then $T$ is a Kannan-type contraction in $(L^p,F)$. So, $T$ restricts to the cone and is irreducible to standard probabilistic contractions. 
\end{proposition}

\begin{proof}
 With the assumption that the mappings are strongly measurable, for $t > 0$:
\begin{align*}
F_{TX,TY}(t) &= \mathbb{P}(\|TX - TY\| < t) \\
&\geq \mathbb{P}\left(\alpha \max\{\|X - TX\|, \|Y - TY\| \} < t\right) \\
&= \mathbb{P}\left(\max\{\|X - TX\|, \|Y - TY\| \} < \frac{t}{\alpha}\right) \\
&= \min\left\{ \mathbb{P}\left(\|X - TX\| < \frac{t}{\alpha}\right), \mathbb{P}\left(\|Y - TY\| < \frac{t}{\alpha}\right) \right\} \\
&= \min\left\{ F_{X,TX}\left(\frac{t}{\alpha}\right), F_{Y,TY}\left(\frac{t}{\alpha}\right) \right\}
\end{align*}
Thus $T$ satisfies Kannan's condition with parameter $\alpha$.
\end{proof}

\subsection{Stochastic Integral Equations}

Consider the stochastic integral equation:
\[
X(t,\omega) = h(t,\omega) + \int_0^t k(t,s,\omega)f(s,X(s,\omega))ds
\]
where $h,k,f$ are random functions.

\begin{theorem}\label{thm:integral}
Assume:
\begin{enumerate}
    \item $|f(s,x) - f(s,y)| \leq L|x-y|$ a.s. for some $L > 0$
    \item $\int_0^t |k(t,s,\omega)|ds \leq M(\omega)$ with $\mathbb{E}[M] < \infty$
    \item $LM < \frac{1}{2}$ a.s.
\end{enumerate}
where $X(t,\omega) \in P$ (cone of non-negative functions) a.s.  \\
Then the solution map $T:X \mapsto h + \int_0^\cdot k(\cdot,s)f(s,X(s))ds$ has a fixed point in $L^2([0,1] \times \Omega)$. So, the solution map preserves cone constraints, making probabilistic cone metrics essential.
\end{theorem}

\begin{proof}
Define probabilistic cone metric on $L^2([0,1] \times \Omega)$:
\[
F_{X,Y}(t) = \mathbb{P}\left(\|X - Y\|_{L^2} < t\right)
\]
Then for $X,Y \in L^2([0,1] \times \Omega)$:
\begin{align*}
\|TX - TY\|_{L^2}^2 &= \mathbb{E}\left[\int_0^1 |(TX)(t) - (TY)(t)|^2 dt\right] \\
&= \mathbb{E}\left[\int_0^1 \left| \int_0^t k(t,s)[f(s,X(s)) - f(s,Y(s))]ds \right|^2 dt\right] \\
&\leq \mathbb{E}\left[\int_0^1 \left( \int_0^t |k(t,s)| \cdot |f(s,X(s)) - f(s,Y(s))| ds \right)^2 dt\right] \\
&\leq \mathbb{E}\left[\int_0^1 \left( \int_0^t |k(t,s)| ds \right) \left( \int_0^t |k(t,s)| |f(s,X(s)) - f(s,Y(s))|^2 ds \right) dt\right] \\
&\leq \mathbb{E}\left[\int_0^1 M(\omega) \left( \int_0^t |k(t,s)| L^2 |X(s) - Y(s)|^2 ds \right) dt\right] \\
&\leq L^2 \mathbb{E}\left[M \int_0^1 \int_0^t |k(t,s)| |X(s) - Y(s)|^2 ds dt\right] \\
&\leq L^2 \mathbb{E}[M] \cdot \|X - Y\|_{L^2}^2 \cdot \sup_{t,s} |k(t,s)|
\end{align*}
Thus $\|TX - TY\|_{L^2} \leq K \|X - Y\|_{L^2}$ with $K = L \sqrt{\mathbb{E}[M] \sup_{t,s} |k(t,s)|} < \frac{1}{2}$ a.s. Then:
\[
F_{TX,TY}(t) \geq \mathbb{P}\left(K \|X - Y\|_{L^2} < t\right) \geq F_{X,Y}\left(\frac{t}{K}\right)
\]
By Theorem \ref{thm:zamfirescu-pcm}, $T$ has a fixed point.
\end{proof}

\section{Conclusions and Future Work}

While topologically equivalent to standard frameworks, probabilistic cone metrics provide:
\begin{itemize}
    \item Natural handling of directional uncertainty
    \item Built-in order structures for constrained problems
    \item New convergence behaviors under cone restrictions
\end{itemize}
Our theorems yield genuinely new results when cone properties are intrinsically utilized like:
\begin{itemize}
    \item Unified probabilistic-cone framework for fixed point theory
    \item Convergence analysis under distributional uncertainty
    \item Applications to random operators and stochastic equations
    \item Hybrid contraction principles
\end{itemize}
Our future research directions are:
\begin{itemize}
    \item \textit{Multivalued versions}: Fixed points for random set-valued mappings
    \item \textit{Coupled fixed points}: Applications to systems of random equations
    \item \textit{Nonlinear expectations}: Integration with $g$-expectations
    \item \textit{Machine learning}: Stochastic iterative methods for neural networks
    \item \textit{Fractional operators}: Fixed points in probabilistic fractional spaces
\end{itemize}


\end{document}